\documentclass[12pt,reqno]{article}
\usepackage{amsfonts, amsthm, amsmath, amssymb}
\usepackage[english]{babel}
\usepackage[cp1251]{inputenc}

\theoremstyle{plain}

\numberwithin{equation}{section}

\newtheorem {lemma} {Lemma}
\newtheorem {theorem} {Theorem}
\newtheorem{proposition}{Proposition}
\theoremstyle{remark}

\begin{document}

\title{The central limit theorem for extremal characters of the infinite symmetric group}

\author {Alexey Bufetov }
\date{}
\maketitle

\begin{abstract}
The asymptotics of the first rows and columns of random Young diagrams corresponding to
extremal characters of the infinite symmetric group is studied. We consider rows and columns
with linear growth in $n$, the number of boxes of random diagrams,
and prove the central limit theorem for them in the case
of distinct Thoma parameters. We also establish a more precise statement relating the growth
of rows and columns of Young diagrams to a simple independent random sampling model.
\end{abstract}

\section{Introduction}

Let $\mathbb Y_n$ be the set of all Young diagrams with $n$ boxes. By $\mathbb Y$ denote the
$\mathbb Z_{\geq 0}$-graded graph, whose $n$-th level is $\mathbb Y_n$, and whose
edges join Young diagrams $\lambda \in \mathbb Y_n$ and $\mu \in \mathbb Y_{n+1}$ if
they differ by exactly one box (then we write $\lambda \uparrow \mu$). Given a diagram
$\lambda$, let $\dim \lambda$
be the number of distinct shortest paths from the one-box diagram to $\lambda$.

A sequence $ \{ M_n \}_{n=1}^{\infty}$
of probability measures on the sets $\mathbb Y_n$ is called a {\it coherent system of distributions} if

\begin{equation*}
M_n(\mu)= \sum_{\lambda : \mu \uparrow \lambda} \frac{\dim \mu}{\dim \lambda} M_{n+1} (\lambda), \ \mbox{ for any } \ \mu \in \mathbb Y_n.
\end{equation*}

It is well-known that the characters of the infinite symmetric group are in one-to-one correspondence
with the coherent systems of distributions on $\mathbb Y$. By Thoma's theorem (see \cite{Th})
the extremal characters can be parameterized by the elements of the set
$\mathcal P = (\{ \alpha_i \}, \{ \beta_j \}, \gamma)$,
where $\alpha_i, \beta_j, \gamma \in \mathbb R$ satisfy
\begin{align*}
& \alpha_1 \ge \alpha_2 \ge \alpha_3 \ge \ldots \ge 0,
\ \ \beta_1 \ge \beta_2 \ge \ldots \ge 0, \\
& \gamma \ge 0, \ \
\sum_{i=1}^{\infty} (\alpha_i + \beta_i) + \gamma =1.
\end{align*}

Let $\{ M_n^{\mathcal P} \}$ be the coherent system of distributions corresponding
to a fixed collection of parameters $\mathcal P$. By $\lambda_i^{\mathcal P} (n)$
(resp.$\lambda_j^{' \mathcal P} (n)$ ) denote
the length of $i$-th row (resp. $j$-th column) of the random Young diagram
distributed according to $M_n^{\mathcal P}$. Our goal is to study the asymptotic behaviour
of these random variables.

It is known (see \cite{KV},\cite{KOO},\cite{KOV}) that $\lambda_i^{\mathcal P} (n)$,
$\lambda_j^{' \mathcal P} (n)$ satisfy the law of large numbers:

\begin{equation*}
\frac{\lambda_i^{\mathcal P} (n)}{n} \xrightarrow[prob]{} \alpha_i, \ \ \
\frac{\lambda_j^{' \mathcal P} (n)}{n} \xrightarrow[prob]{} \beta_j.
\end{equation*}

The central limit theorem was first established by F\'eray and M\'eliot \cite{FM}
for the case $\alpha_i=(1-q) q^{i-1}$, $\beta_j=0$, $\gamma=0$.
In the present paper we prove the central limit theorem
for the case of strictly monotone sequences of parameters.
In the recent paper by M\'eliot \cite{M} this theorem was also proved by a different method.

Fix a set of parameters
$\mathcal P$ and assume that
\begin{equation}
\label{monot}
\alpha_i > \alpha_{i+1}, \ \mbox{ for all } i : \alpha_i \ne 0; \ \ \ \
\beta_j > \beta_{j+1}, \ \mbox{ for all } j : \beta_j \ne 0.
\end{equation}

\begin{theorem}[\bf Central limit theorem]
Let $\mathcal P$ be any set of parameters satisfying \eqref{monot} and let $K,L>0$
be integers such that $\alpha_1 > \alpha_2 > \dots > \alpha_K >0$,
$\beta_1 > \beta_2 > \dots > \beta_L >0$. Then
\begin{multline*}
\Bigl(\frac{ \lambda_1^{\mathcal P} (n) - \alpha_1 n}{ \sqrt{n}}, \frac{ \lambda_2^{\mathcal P} (n) - \alpha_2 n}{ \sqrt{n}}, \ldots,
\frac{ \lambda_K^{\mathcal P} (n) - \alpha_K n}{ \sqrt{n}}, \frac{ \lambda^{' \mathcal P}_1 (n) - \beta_1 n}{ \sqrt{n}}, \\ \ldots,
\frac{ \lambda^{' \mathcal P}_L (n) - \beta_L n}{ \sqrt{n}} \Bigr) \xrightarrow[Law]{} Z=(Z_1, Z_2, \ldots, Z_k, Z_1^{'}, \ldots, Z_L^{'}),
\end{multline*}
where $Z$ is a multidimensional Gaussian random variable with moments:
\begin{align*}
& \mathbf E Z_i =0, \ \ \ \mathbf E Z'_i =0,  \\
& \mathbf E Z^2_i =\alpha_i - \alpha_i^2, \ \ \ \mathbf E Z'^2_i = \beta_i - \beta_i^2,  \\
& \mathbf E Z_i Z_j =-\alpha_i \alpha_j, \ \ \  \mathbf E Z'_i Z'_j = - \beta_i \beta_j, \ \ \ \mathbf E Z_i Z'_j = -\alpha_i \beta_j.
\end{align*}
\end{theorem}

For the case $\alpha_1 = \dots = \alpha_k = \frac{1}{k}$ the fluctuations are not Gaussian (see \cite[Ch.3.3 Th.2] {Ker},
\cite[Th.1.6]{Joh}),
so the assumption \eqref{monot} is essential.

{\it Remark 1}. For all $\alpha_i \ne 0$, $\beta_j \ne 0$ let  $\{ X_i \}$, $\{Y_j \}$,$\Theta$
be independent Gaussian random variables with zero mean and dispersions:

\begin{equation*}
\mathbf{E} X_i^2 = \alpha_i, \ \ \mathbf{E} Y_j^2=\beta_j, \ \ \ \mathbf{E} \Theta^2 = \gamma.
\end{equation*}
Then the distribution $(Z_1, \dots, Z_K, Z_1^{'}, \dots, Z_L{'})$ coincides with
the projection to the first $K+L$ coordinates of the
conditional distribution on the
hyperplane

\begin{equation*}
X_1+ \dots + X_K+X_{K+1}+ \dots +Y_1+\dots + Y_L+ Y_{L+1}+ \dots + \Theta =0.
\end{equation*}

{\it Remark 2}. Let $\widetilde M_{\nu}^{\mathcal P}$ stand for
the measure on $\mathbb Y$ defined by the formula

\begin{equation*}
\widetilde M_{\nu}^{\mathcal P} (\lambda) :=
e^{-\nu} \frac{\nu^{|\lambda|}}{|\lambda|!} M_{|\lambda|}^{\mathcal P} (\lambda),
\end{equation*}
where $|\lambda|$ is the number of boxes in $\lambda$.
The measure $\widetilde M_{\nu}^{\mathcal P}$ is called the {\it poissonization} of measures $\{M_n^{\mathcal P} \}$
with parameter $\nu>0$.
Let $\tilde \lambda_i^{\mathcal P} (\nu)$(resp. $\tilde \lambda_j^{' \mathcal P} (\nu)$) be the length
of $i$-th row (resp. $j$-th column) of the random Young diagram distributed according to the measure
$\widetilde M_{\nu}^{\mathcal P}$.
Under the assumptions of Theorem 1 we have
\begin{multline*}
\Bigl(\frac{\tilde \lambda_1^{\mathcal P} (\nu) - \alpha_1 \nu}{ \sqrt{\nu}}, \frac{\tilde \lambda_2^{\mathcal P} (\nu) - \alpha_2 \nu}{ \sqrt{\nu}}, \ldots,
\frac{\tilde \lambda_K^{\mathcal P} (\nu) - \alpha_K \nu}{ \sqrt{\nu}}, \frac{\tilde \lambda^{' \mathcal P}_1 (\nu) - \beta_1 \nu}{ \sqrt{\nu}}, \\ \ldots,
\frac{ \tilde \lambda^{' \mathcal P}_L (\nu) - \beta_L \nu}{ \sqrt{\nu}} \Bigr) \xrightarrow[Law]{\nu \to \infty} (X_1, X_2, \ldots, X_K, Y_1, \ldots, Y_L).
\end{multline*}
The proof is similar to the proof of Theorem 1.

Let $\mathcal A=L_e \cup L_o \cup G$ be an alphabet, $L_e = \{ x_1, x_2, \ldots \}$ and
$L_o= \{ y_1, y_2, \ldots \}$ are discrete parts of $\mathcal A$ and let $G$ be
the continuous part which we will identify with an interval in $\mathbb R$.
Let $\mu_1$ be the probability measure on $\mathcal A$
such that $\mu_1(\{ x_i \})=\alpha_i$, $\mu_1(\{ y_j \})=\beta_j$, and the restriction of $\mu_1$ to $G$
is proportional to Lebesque measure on $G$ with total mass $\mu_1(G)=\gamma$. Let $\mu_n = \mu_1^{\otimes n}$ be
the product Bernoulli measure on $\mathcal A^n$. By $N_{x_i}(n)$ (resp. $N_{y_j}(n)$) we denote the
number of letters $x_i$ (resp. $y_j$) in the random word $w \in \mathcal A^n$ distributed according to the measure
$\mu_n$. Let $p$ be a linear order on $\mathcal A$. It was shown in \cite{vershik-kerov}
that a certain generalisation of the RSK-algorithm provides a map

\begin{equation*}
\phi_p : \mathcal A^n \to \mathbb Y_n
\end{equation*}
such that

\begin{equation*}
\phi_p(\mu_n) = M_n^{\mathcal P}.
\end{equation*}

Therefore $\lambda_i^{\mathcal P} (n), \lambda_j^{' \mathcal P} (n)$
can be defined as functions on the probability space $(\mathcal A^n, \mu_n)$.

\begin{theorem}
Let $\mathcal P$ be any set of parameters satisfying \eqref{monot} and
let $K,L >0$ be integers such that $\alpha_1 > \alpha_2 > \dots > \alpha_K >0$
and $\beta_1 > \beta_2 > \dots > \beta_L >0$. Let
\begin{align*}
& \epsilon_1 (n) := \lambda_1^{\mathcal P} (n) - N_{x_1} (n),  \\
& \epsilon_2 (n) := \lambda_2^{\mathcal P} (n) - N_{x_2} (n),  \\
& \vdots \\
& \epsilon_K (n) := \lambda_K^{\mathcal P} (n) - N_{x_K} (n),  \\
& \epsilon'_1 (n) := \lambda_1^{' \mathcal P} (n) - N_{y_1} (n),   \\
& \vdots \\
& \epsilon'_L (n) := \lambda_L^{' \mathcal P} (n) - N_{y_L} (n).
\end{align*}
Then there exists a constant $C=C(K,L)$ such that
\begin{equation*}
\mathbf E |\epsilon_i (n) | < C, \ \ \ \mathbf E |\epsilon'_j(n)| < C, \ \ \ i=1 \ldots K, \ \ \ j=1 \ldots L.
\end{equation*}
\end{theorem}

Theorem 1 can be easily derived from Theorem 2.

{\it Remark 3}.
There are several ways to define the random variables $\lambda_i^{\mathcal P} (n)$ on the probability
space $(\mathcal A^n, \mu_n)$ corresponding to different linear orders on $\mathcal A$.
It will be shown in section 2.2 that the distribution of $\lambda_i^{\mathcal P} (n) - N_{x_i}(n)$
does not depend on a specific choice of a linear order on $\mathcal A$.

{\it Remark 4}.
Let $\tilde \mu_{\nu} (w)$ be the measure on the set of all finite words (with letters from alphabet $\mathcal A$)
defined by the formula
\begin{equation*}
\tilde \mu_{\nu} (w) := e^{-\nu} \frac{\nu^{|w|}}{|w|!} \mu_{|w|} (w),
\end{equation*}
where $|w|$ is the number of letters in $w$.
Let $\tilde N_{x_i} (\nu)$ (resp. $\tilde N_{y_j} (\nu)$ ) be the number of letters $x_i$ (resp. $y_j$) in
the random word $w$ distributed according to the measure $\tilde \mu_{\nu}$.
Under the same assumptions the statement of Theorem 2 holds for the random variables
$\tilde \lambda_i^{\mathcal P} (\nu) - \tilde N_{x_i} (\nu)$,
$\tilde \lambda_j^{' \mathcal P} (\nu) - \tilde N_{y_j} (\nu)$.
The proof is similar to the proof of Theorem 2.

{\it Remark 5}.
Theorem 2 can be restated in terms that do not use the RSK-algorithm.
Then variables
$\lambda_i^{\mathcal P} (n), \lambda_j^{' \mathcal P}(n)$ should be
defined on the probability space of $\mathcal A_p$-tableaux (see
the definition in section 2.1).

{\bf Acknowledgments}. I am deeply grateful to G.Olshanski and A.Borodin for the
statement of the problem and for
numerous helpful discussions. I am grateful to L.Petrov for valuable remarks.
I was partially supported by Simons Foundation--IUM scholarship,
by Moebius Foundation scholarship, and by RFBR-CNRS Grant 10-01-93114.

\section {Main lemmas}
\subsection{}

In this section we recall some facts from \cite{vershik-kerov}.

Let $p$ be a linear order on $\mathcal A$. For $x,y \in \mathcal A$ we will write $x \nearrow y$
if $x<y$ or $x=y \in L_e$ and $x \searrow y$ if $x>y$ or $x=y \in L_o \cup G$. A word
$w=x_1 x_2 \dots x_n$ is called {\it increasing} if $x_1 \nearrow x_2 \dots \nearrow x_n$ and
{\it decreasing} if $x_1 \searrow x_2 \dots \searrow x_n$.

A Young diagram $\lambda$ filled by letters from $\mathcal A$ is called an {\it $\mathcal A_p$-tableau}
of shape $\lambda$ if the letters are increasing along the rows and decreasing along the columns
if we read them from bottom to top (see the example below).

For each $w \in \mathcal A^n$ the generalised RSK-algorithm produces a pair $(R(w),S(w))$, where
$R(w)$ is an $\mathcal A_p$-tableau, $S(w)$ is a standard\footnote{A standard Young
tableau is a diagram $\lambda$ filled by numbers from 1 to $|\lambda|$, each occuring once;
the numbers form an increasing sequence along each row and down each column.}
Young tableau, and $R(w)$, $S(w)$ have
the same shape $\lambda$. Let the map

\begin{equation*}
\phi_p \colon \mathcal A^n \to \mathbb Y_n
\end{equation*}
take each $w$ to the shape of $R(w)$ and $S(w)$.

We recall the definition of the generalised RSK-algorithm. At first, we define the algorithm of row bumping;
given an $\mathcal A_p$-tableau $T$ and a letter $x \in \mathcal A$, it produces a new $\mathcal A_p$-tableau,
denoted $x \to T$. This tableau will have one more box than $T$, and its entries will be those of $T$, together
with one more entry labelled $x$. Suppose $x \in L_e$. If $x$ is greater or equal to all the entries in the first
row of $T$, we add $x$ in a new box to the end of the first row. Otherwise we find the left-most entry in the
first row that is strictly greater than $x$, replace it with $x$ and bump the old entry.
If $x \in L_o$, then the rule is the same, but $x$ can bump not only larger entries, but equal entries also.
Take this entry that was bumped from the first row, and repeat the process for the second row.
Keep going until the bumped entry can be put at the end of the row it is bumped into, or until it
is bumped out of the bottom row, in which case it forms a new row with one entry.

For $w=x_1 x_2 \dots x_n$ define $R(w)$ by the formula

\begin{equation*}
R(w):=[( x_n \to (x_{n-1} \to (x_{n-2} \dots \to (x_2 \to (x_1 \to \emptyset )) \dots )].
\end{equation*}
On each step of the algorithm one new box joins $R(w)$. Let us put number $i$ to
the box from the $i$-th step; by definition, $S(w)$ is the standard Young tableau thus obtained.

{\it Example}. Let $x_1<x_2<y_1<y_2$ and let $L_e=\{x_1,x_2\}$, $L_o=\{y_1,y_2\}$.
Then the algorithm transforms the word $w= x_1 y_1 y_1 y_2 x_2 x_1 y_1$ to the pair of
tableaux

\begin{picture}(100,100)
\put(5,5){$y_1$}
\put(5,25){$y_1$}
\put(5,45){$x_2$}
\put(5,65){$x_1$}
\put(25,65){$x_1$}
\put(25,45){$y_2$}
\put(45,65){$y_1$}
\put(0,0){\line(0,1){80}}
\put(0,80){\line(1,0){60}}
\put(60,80){\line(0,-1){20}}
\put(60,60){\line(-1,0){20}}
\put(40,60){\line(0,-1){20}}
\put(40,40){\line(-1,0){20}}
\put(20,40){\line(0,-1){40}}
\put(20,0){\line(-1,0){20}}
\put(0,60){\line(1,0){40}}
\put(0,40){\line(1,0){20}}
\put(0,20){\line(1,0){20}}
\put(20,80){\line(0,-1){40}}
\put(40,80){\line(0,-1){20}}
\end{picture}
\begin{picture}(100,100)
\put(5,5){$6$}
\put(5,25){$5$}
\put(5,45){$3$}
\put(5,65){$1$}
\put(25,65){$2$}
\put(25,45){$7$}
\put(45,65){$4$}
\put(0,0){\line(0,1){80}}
\put(0,80){\line(1,0){60}}
\put(60,80){\line(0,-1){20}}
\put(60,60){\line(-1,0){20}}
\put(40,60){\line(0,-1){20}}
\put(40,40){\line(-1,0){20}}
\put(20,40){\line(0,-1){40}}
\put(20,0){\line(-1,0){20}}
\put(0,60){\line(1,0){40}}
\put(0,40){\line(1,0){20}}
\put(0,20){\line(1,0){20}}
\put(20,80){\line(0,-1){40}}
\put(40,80){\line(0,-1){20}}
\end{picture}

Denote the maximal cardinality of a disjoint union of $k$ increasing (resp. decreasing) subsequences
of the word $w$ by $r_k(w)$ (resp. $c_k(w)$).

\begin{proposition}

a)
The generalised RSK-algorithm provides a bijection between
$\mathcal A^n$ and the set of pairs $(R,S)$, where $R$ is an
$\mathcal A_p$-tableau, $S$ is a standard Young tableau, and $R$,$S$ have
the same shape consisting of $n$ boxes.

b)
The following relations hold
\begin{equation*}
r_k (w) = \sum_{i=1}^k \lambda_i (\phi_p(w)); \ \ \ \
c_k (w) = \sum_{j=1}^k \lambda'_j(\phi_p(w)).
\end{equation*}

\end{proposition}

\begin{proof}
This proposition is a generalisation of Shensted's theorem (see \cite{Sh}).
In \cite[Prop.1]{vershik-kerov} it was pointed out that the proof is
analogous to the proof of Shensted's theorem (see, e.g., \cite{fulton}).
\end{proof}

Let $\Lambda$ be the algebra of symmetric functions in infinitely many variables
(see \cite[Ch. 1.2]{macdonald}). Let $h_n$ be the complete homogeneous symmetric
functions and let $s_{\lambda}$ be the Schur functions. Define the generating
function of $\{ h_n \}$ by the rule

\begin{equation*}
H(z)=1+ \sum_{n=1}^{\infty} h_n z^n
\end{equation*}
and let
\begin{equation*}
\pi^{\mathcal P} \colon \Lambda \to \mathbb C
\end{equation*}
be the homomorphism defined by the formula

\begin{equation*}
\pi^{\mathcal P} (H(z)) = e^{\gamma z} \prod_{i \ge 1} \frac {1+ \beta_i z}{1- \alpha_i z}.
\end{equation*}

\begin{proposition}

a) Let $P_{\mathcal P} (\lambda)$ denote the probability that the random
filling of the diagram $\lambda$ with independent letters from $\mathcal A$ with
common distribution $\mu_1$ produces an $\mathcal A_p$-tableau. Then

\begin{equation*}
P_{\mathcal P} (\lambda) = \pi^{\mathcal P} (s_{\lambda}).
\end{equation*}

b) Suppose $\lambda \in \mathbb Y_n$. Then
\begin{equation*}
\mu_n (w : \phi_p (w) = \lambda) = \pi^{\mathcal P} (s_{\lambda}) \dim \lambda = M_n^{\mathcal P} (\lambda).
\end{equation*}
\end{proposition}

\begin{proof} See \cite[Prop.3 and Th.1]{vershik-kerov}. \end{proof}

There are other similar generalisations of the RSK-algorithm (see \cite{berele-regev},
\cite{regev-seeman}); they possess the properties stated in Proposition 1a and
Proposition 2b but not that of Proposition 1b.

\subsection{}

A set $I \subset \mathcal A$ is called an {\it interval} if for any $a_1,a_2 \in I$
and $a \in \mathcal A$ such that

\begin{equation*}
a_1 < a < a_2
\end{equation*}
it follows that $a \in I$.
For the sake of convenience, in what follows we consider
only those linear orders on
$\mathcal A$ for which $G$ is an interval.

By $n_i(R)$ (resp. $n'_j (R)$) denote the number of letters $x_i$ (resp. $y_j$) in
an $\mathcal A_p$-tableau $R$. Let $m(R)$ be the number of letters from $G$ in $R$. The
collection of numbers $(\{ n_i(R) \},\{ n'_j(R) \},m(R))$ is called the {\it type} of
$\mathcal A_p$-tableau $R$ and is denoted by $type(R)$.

We  recall that different linear orders on $\mathcal A$ produce different
maps
\begin{equation*}
\phi_p \colon \mathcal A^n \to \mathbb Y_n.
\end{equation*}

\begin{lemma}
Suppose a collection of numbers $(\{ n_i \} ; \{ n'_j\}; m)$ is fixed, and
take a diagram $\lambda \in \mathbb Y_n$. Then the probability
\begin{equation*}
\mu_n(w: \phi_p(w) = \lambda ; type(R(w))=(\{n_i \}; \{ n'_j \};m))
\end{equation*}
does not depend on a specific choice of order $p$.
\end{lemma}

\begin{proof}
Note that the probability of coincidence of two letters from $G$ in $w$ is equal to 0;
hence it can be assumed that all letters from $G$ in $w$ are pairwise distinct. Let $g_1< \dots <g_m$
be arbitrary letters from $G$. Consider a collection of $n$ letters
$\Omega = (\{x_i \}, \{ y_j \}, g_1, \dots, g_m )$, where $x_i$ appears $n_i$ times, $y_j$
appears $n'_j$ times. Consider various fillings of $\lambda$ by all letters of $\Omega$ forming
an $\mathcal A_p$-tableau. It follows from \cite[Th.3]{regev-seeman}
that the number of such fillings does not depend on a specific choice of order $p$. Denote this number by
$d( \{ n_i \}; \{ n'_j \};m)$. By virtue of Proposition 1a, there are
exactly $\dim\lambda$ words composed from
the letters of $\Omega$ that are associated with every such filling.
Therefore
the probability of every such filling of diagram $\lambda$ is equal to

\begin{equation*}
\dim \lambda \prod_{i \ge 1} \alpha_i^{n_i} \prod_{j \ge 1} \beta_j^{n'_j} \frac{1}{m!},
\end{equation*}
where the factor $\frac{1}{m!}$ comes from the condition $g_1 < g_2 < \dots < g_m$.
Consequently,

\begin{multline*}
\mu_n(w: \phi_p(w) = \lambda ; type(R(w))=(\{n_i\};\{n'_j\};m)) = \\ \dim \lambda \frac{d(\{n_i\};\{n'_j\};m)}{m!}
\prod_{i \ge 1} \alpha_i^{n_i} \prod_{j \ge 1} \beta_j^{n'_j},
\end{multline*}
where the right-hand side does not depend on $p$.
\end{proof}

{\bf Corollary}.
The distributions of $\lambda_i^{\mathcal P} (n) - N_{x_i} (n)$,
$\lambda_j^{' \mathcal P} (n) - N_{y_j} (n)$ do not depend on $p$.

Let us fix an order $p$ on $\mathcal A$ and let $I$ be an interval of $\mathcal A$
with respect to $p$.
Consider a new alphabet $\mathcal A^{*}$ obtained from $\mathcal A$ by shrinking $I$ to a single new letter $z$,
all letters of $\mathcal A\setminus I$ remainng unchanged.
We assume that $z\in L_e$ and $\mu_1(\{z\})=\mu_1(I)$.
We say that $\mathcal A^{*}$ is an {\it amalgamation} of $\mathcal A$.

Recall that the map $\phi_p$ is defined on the probability space $(\mathcal A^n, \mu_n)$; in this
case the map $\phi_p^{*}$ can be naturally defined as

\begin{equation*}
\phi^*_p \colon \mathcal A^n \to \mathbb Y_n,
\end{equation*}
i.e. on the same probability space. Thus we can compare lengths of rows
of random Young diagrams generated by measures on $\mathcal A$ and $\mathcal A^{*}$.

\begin{lemma}
For any $k>0$ the following inequality holds
\begin{equation*}
\sum_{i=1}^k \lambda_i^{\mathcal P} (n) \le \sum_{i=1}^k \lambda_i^{\mathcal P^*} (n).
\end{equation*}

\end{lemma}

\begin{proof}
By Proposition 1b it follows that

\begin{equation*}
\sum_{i=1}^k \lambda_i^{\mathcal P} (n) = r_k(w);
\end{equation*}
\begin{equation*}
\sum_{i=1}^k \lambda_i^{\mathcal P^*} (n) = r_k(w^*).
\end{equation*}

Notice that any increasing (in the sense of section 2.1) subsequence of
a word $w \in \mathcal A^n$ turns into the increasing subsequence of the
corresponding word $w^* \in \mathcal A^{*n}$ because the new letter $z$
belongs to the set $L_e$. Hence for any $w$ we have

\begin{equation*}
r_k(w^*) \ge r_k(w).
\end{equation*}
\end{proof}

Let $p^t$ be the order on $\mathcal A$ which is inverse to $p$. Assume that
\begin{equation*}
L_e^t = L_o, \ \ \
L_o^t = L_e.
\end{equation*}
Define the {\it transposed} map

\begin{equation*}
\phi_{p^t} \colon \mathcal A^n \to \mathbb Y_n
\end{equation*}
by the generalised RSK-algorithm applied to the order $p^t$ and the sets $L_e^t, L_o^t, G$.
This map changes the roles of rows and columns of Young diagram or,
equivalently, the roles of parameters $\{\alpha_i \}$ and $\{\beta_j\}$.

By $\lambda^t$ we denote the Young diagram which is the transpose
of the Young diagram $\lambda$.

\begin{lemma}
\begin{equation*}
\phi_p(w) = \phi_{p^t}(w)^t \mbox{ for almost all } w.
\end{equation*}
\end{lemma}
\begin{proof}
Suppose all letters from $G$ in $w$ are distinct (this condition is necessary
because the relations $x_1 \nearrow x_2$ and $x_1 \searrow x_2$ are not formally
symmetric); then any increasing subsequence with respect to $p$ and $L_e \cup L_o$
is a decreasing subsequence with respect to $p^t$ and $L_e^t \cup L_o^t$.
Therefore the lemma follows from Proposition 1b.
\end{proof}

\subsection{}

Let $q_1,q_2,q_3 \geq 0$, $q_1 < q_3$, and $q_1+q_2+q_3 = 1$.
Consider a random walk on the set $\{0,1,2 \dots \}$ under which
a particle goes right with probability $q_1$ and left
with probability $q_3$ (except for the point 0). In the first moment the
particle is in 0. By $\Psi_{q_3,q_1}(n)$ denote the position of
particle after $n$ steps. In the other words $\Psi_{q_3,q_1}(n)$ is a
Markovian chain with the transition matrix

\begin{equation*}
D=\begin{pmatrix}
q_3+q_2 & q_1 & 0 & 0 & \ldots & \ldots \\
q_3 & q_2 & q_1 & 0 & 0 & \ldots \\
0 & q_3 & q_2 & q_1 & 0 & \ldots \\
\vdots & \ddots & \ddots & \ddots & \ddots & \ddots
\end{pmatrix}
\end{equation*}
and the initial vector $\vec{a}_0 = (1,0,0,0, \ldots)$.

\begin{lemma}
There exists a constant $C$ such that
\begin{equation*}
\mathbf E \Psi_{q_3,q_1}(n) < C \ \ \ \ \mbox{ for all n}.
\end{equation*}
\end{lemma}

\begin{proof}
By definition, put
\begin{equation*}
\vec{a} := (2,2(\frac{q_1}{q_3}) ,2(\frac{q_1}{q_3})^2, \ldots ).
\end{equation*}
Clearly, $ \vec{a} D = \vec{a}$. Besides, each component of $\vec{a}$ is larger
than the corresponding component of $\vec{a}_0 = (1,0,0 \dots )$. It follows
that the components of $\vec{a} D^n$ are larger than the components of $\vec{a}_0 D^n$
for any $n$. Hence $\mathbf E \Psi_{q_3,q_1}(n)$
is bounded by the number
$$
2 \sum_{i=0}^{\infty} i (\frac{q_1}{q_3})^i < \infty.
$$
\end{proof}

\subsection{}
Let us fix an order $p$ on $\mathcal A$. Let $a,b \in L_e$, $a<b$ such that
$\{ a, b \}$ is an
interval with respect to $p$ (i.e. $a$ and $b$ are neighbours), and $w \in \mathcal A^n$.
Cross out from $w$ all letters except $a$ and $b$; by $w_{a,b}$ denote the remaining
word.

Consider the action of generalised RSK-algorithm on the word $w$.
Form a new word by writing letters $a$,$b$ in order of bumping them from the first row
of $\mathcal A$-tableau; if some letters $a$ and $b$ remain in the first row then we write
them in the end of this new word. We say that this new word is a {\it possible transformation}
of the word $w_{a,b}$ and denote it by $d_w(w_{a,b})$.

For $w=z_1 z_2 \dots z_n$ any subword of the form $z_k z_{k+1} \dots z_n$ is called a {\it suffix}.
For all suffixes (including the empty one) of the word $w_{a,b}$ we compute the differencies between
the number of letters $b$ and the number of letters $a$ in these suffixes. The
maximal among differencies is called the {\it result} of $w_{a,b}$ and is denoted by $\rho(w_{a,b})$.
If the difference between numbers of $b$ and $a$ is equal to $\rho(w_{a,b})$ in some suffix
then we say that this suffix is {\it maximal}.

It is easily shown that if we apply the generalised RSK-algorithm directly to the word $w_{a,b}$
then there will be exactly $\rho(w_{a,b})$ letters $b$ in the first row.

{\it Example}. Suppose $x_1<x_2<x_3$ and $w= x_2 x_1 x_3 x_2 x_1 x_2 x_3 x_3 x_2 x_3 x_1 x_3 x_2$.
Then $w_{x_2 x_3} = x_2 x_3 x_2 x_2 x_3 x_3 x_2 x_3 x_3 x_2$, the maximal suffix of $w_{x_2,x_3}$
consists of the last 6 letters, and $\rho( w_{x_2,x_3})=2$.
Also we have $$d_w(w_{x_2,x_3})= x_2 x_3 x_2 x_3 x_2 x_3 x_2 x_2 x_3 x_3.$$

\begin{lemma}
For all $w \in \mathcal A^n$ the following is true
\begin{equation*}
\rho (d_w(w_{a,b})) \leq \rho (w_{a,b}).
\end{equation*}
\end{lemma}

\begin{proof}

{\it Step 1}

We will arrange some letters of $w_{a,b}$ in pairs.
In each pair there will be one letter
$b$ and one letter $a$ which is to the right of this letter $b$.
To construct the first pair we take the right-most letter $a$ in $w_{a,b}$ and pair it
with the first letter $b$ to the left of this letter $a$.
To construct the $k$-th pair we take the $k$-th letter $a$
from the right and pair it with the first letter $b$ which is to the left
of our letter $a$ and have not
chosen yet. We will make this procedure as many times as possible.
We will call a letter $b$ 'white' if it is in a pair with some
letter $a$ and 'black' otherwise.

For the word $w_{x_2,x_3}$ from example (see above) the pairs will be the following:

\begin{picture}(180,30)
\put(0,0){$x_2$}
\put(15,0){$x_3$}
\put(30,0){$x_2$}
\put(45,0){$x_2$}
\put(60,0){$x_3$}
\put(75,0){$x_3$}
\put(90,0){$x_2$}
\put(105,0){$x_3$}
\put(120,0){$x_3$}
\put(135,0){$x_2$}
\qbezier(140,8)(132,20)(125,8)
\qbezier(95,8)(87,20)(80,8)
\qbezier(50,8)(35,20)(20,8)
\end{picture}

{\it Step 2}

Let us prove that exactly $\rho(w_{a,b})$ letters $b$ are not in pairs.
Indeed, denote by $\rho'$ the number of 'black' letters in $w_{a,b}$.
Consider the suffix beginning at the left-most 'black' letter.
It is clear that every letter $a$ from this suffix is paired with
a letter $b$ from this suffix. Therefore $\rho(w_{a,b}) \geq \rho'$.
On the other hand, there should be at least $\rho(w_{a,b})$
'black' letters in a maximal suffix. Thus $\rho'=\rho(w_{a,b})$.

Consider a step of RSK-algorithm when we should bump a letter $b$
from the first row. Let us assume that if there are a 'white' letter $b$
in the first row then we bump it; we bump 'black' letter $b$ only
if all letters $b$ in the first row are 'black'. Evidently, this agreement
does not affect the action of algorithm.

{\it Step 3}

We claim that in each pair the letter $b$ will be bumped earlier than the letter $a$.
Let us prove this by induction on the number of pairs. At first, consider the
left-most pair. Letter $b$ from this pair is the first 'white' letter $b$ in the
word therefore letter $a$ from this pair must bump it. Now consider the $k$-th pair
from the left. In the moment of appearance of letter $a$ from this pair all letters $b$
from the previous $k-1$ pairs must be already bumped (it follows from the inductive hypothesis).
Therefore the arriving letter $a$ must bump letter $b$ from its own pair (if it was not
bumped earlier which is also possible). Hence the order in each pair will be
the same after any possible transformation of the word $w_{a,b}$.

Consequently, only 'black' letters of $w_{a,b}$ can contribute to the result of $d_w( w_{a,b})$.
There are only $\rho(w_{a,b})$ 'black' letters in $w_{a,b}$ therfore $\rho( d_w( w_{a,b})) \leq \rho(w_{a,b})$.
\end{proof}

\section{ Proofs of theorems}

\subsection{ The proof of Theorem 2 for finite  $\mathcal A$}

Consider the special case when the numbers of $\alpha$- and $\beta$-parameters
are finite and $\gamma=0$. Then without loss of generality we can assume that $K$ is
equal to the number of $\alpha$-parameters and $L$ is equal to the number of
$\beta$-parameters. Thus our alphabet becomes
$\mathcal A = \{x_1, x_2, \ldots, x_K \} \cup \{y_1, y_2, \ldots, y_L \}$.

Introduce the linear order on $\mathcal A$:

\begin{equation*}
x_1 < x_2 < \dots <x_K < y_L < y_{L-1} < \dots < y_2<y_1.
\end{equation*}
By corollary of Lemma 1 it suffices to prove the theorem for this order.

We will apply the generalised RSK-algrothm to a word $w \in \mathcal A^n$. By $\xi_j^i(n)$
denote the number of letters $x_j$ in the $i$-th row of $\mathcal A_p$-tableau $R(w)$.
It is easy to see that for our order $\xi_j^i(n)=0$ if $i>j$.

A sequence of random variables $\psi(n)$ is called {\it L-bounded } if there is
a constant $C$ which does not depend on $n$ such that

\begin{equation*}
\mathbf E |\psi(n)| <C  \ \ \ \mbox{ for all } n.
\end{equation*}
By $\{ L(n) \}$ we denote any L-bounded sequence of random variables. Notice
that

\begin{equation*}
\{ L(n) \}+\{ L(n) \} =\{ L(n) \}, \ \ \ \{ L(n) \}- \{L(n) \}= \{ L(n) \}.
\end{equation*}

Let us prove the statement of Theorem 2 for rows. The proof is by induction
on the number of rows.

1)First row

{\it Lower bound}.
Note that all letters $x_1$ are in the first row. Therefore,

\begin{equation*}
\lambda_1^{\mathcal P}(n) \ge N_{x_1}(n).
\end{equation*}

{\it Upper bound}.
Consider the number $\xi_k^1(n)$, $k \geq 2$. It is increased by 1 if $x_k$ appears;
the probability of this event is $\alpha_k$. Suppose $\xi_k^1(n) \ne 0$;
if $x_{k-1}$ appears then $\xi_k^1(n)$ is decreased by 1 with probability
$\alpha_{k-1}>\alpha_k$. By Lemma 4 it follows that the sequence $\xi_k^1(n)$
is L-bounded. In each row we have at most one letter $y_j$ for any $j$.
Therefore,

\begin{equation*}
\lambda_1^{\mathcal P}(n)=N_{x_1}(n)+\{ L(n) \}.
\end{equation*}

2)Fix $l \leq K$. Assume that the theorem holds for the first
$l-1$ rows and prove it for the $l$-th row.

{\it Lower bound}. Letter $x_l$ can not be lower than the $l$-th row.
On the other hand, the number of letters $x_l$ in the first $l-1$ rows
is L-bounded by the inductive hypothesis. Hence,

\begin{equation*}
\lambda_l^{\mathcal P} (n) \geq N_{x_l} (n) - \{ L(n) \}.
\end{equation*}

{\it Upper bound}.
Let $w^i$ be the word consisting of letters which were bumped out
from the $i$-th row and let $w^0=w$.

We will prove that $\xi_k^l(n)$ is L-bounded for $k>l$.
Note that $w^{i-1}_{x_{k-1},x_k}$ is a sequence of letters $x_{k-1}$ and $x_k$
arriving at the $i$-th row and $w^i_{x_{k-1},x_k}$ is a sequence of
letters $x_{k-1}$,$x_k$ bumped from the $i$-th row. Thus $w^i_{x_{k-1},x_k}$
is a possible transformation of $w^{i-1}_{x_{k-1},x_k}$ except for letters $x_{k-1}$
and $x_k$ which remain in the $i$-th row. By the inductive hypothesis there is a
L-bounded number of such letters. Therefore we can apply Lemma 5 and get

\begin{equation}
\label{R<R}
\rho(w^i_{x_{k-1},x_k}) \le \rho(w^{i-1}_{x_{k-1},x_k}) + \{ L(n) \}, \ \ \ i=1 \dots l-1.
\end{equation}
We sum inequalities \eqref{R<R} for $i=1 \dots l-1$ and obtain

\begin{equation*}
\rho(w^{l-1}_{x_{k-1},x_k}) \le \rho(w^0_{x_{k-1},x_k}) + \{L(n)\}.
\end{equation*}
Moreover,
\begin{equation*}
\rho(w^0_{x_{k-1},x_k}) = \{L(n)\}.
\end{equation*}
by Lemma 4.
Arguing as in the proof of upper bound for the first row, we see that

\begin{equation*}
\xi_k^l(n) \le \rho(w^{l-1}_{x_{k-1},x_k}).
\end{equation*}
It follows that $\xi^l_k(n)$ is a L-bounded sequence for any $k>l$.
In each row we have at most one letter $y_j$ for all $j$. Therefore,

\begin{equation*}
\lambda_l^{\mathcal P} (n)= N_{x_l}(n) + \{L(n)\}.
\end{equation*}

For estimating $\lambda_1^{'\mathcal P}(n), \dots, \lambda_L^{'\mathcal P}(n)$
we consider the transposed map. We can apply bounds for rows to the measure
on $\mathbb Y_n$ defined by the parameters
$\mathcal P^{t} = (\{ \beta_j \},\{ \alpha_i \},\gamma)$.
Using Lemma 3, we get

\begin{align*}
& \lambda_1^{'\mathcal P}(n) = \lambda_1^{\mathcal P^t} (n)= N_{y_1}(n) + \{L(n)\},  \\
&  \vdots \\
& \lambda_L^{'\mathcal P}(n) = \lambda_L^{\mathcal P^t} (n)= N_{y_L}(n) + \{L(n)\}.
\end{align*}

\subsection{The proof of Theorem 2 for the general case }

Let $\mathcal P=(\{ \alpha_i \},\{ \beta_j \},\gamma)$ satisfy the assumption of
strict monotonicity \eqref{monot}. First let us prove the theorem for the lengths of rows.

By corollary of Lemma 1, upper and lower bounds for
$\mathbf E (\lambda_i^{\mathcal P} (n) - N_{x_i}(n))$
can be proved for different orders on $\mathcal A$.

Introduce the order $p_1$ on $\mathcal A$:

\begin{equation*}
x_1<x_2< \dots< y_1 < y_2 < \dots <G.
\end{equation*}

{\it Lower bound}.
For the order $p_1$ the evolution of letters $x_1,x_2, \dots, x_K$
does not depend on other letters, therefore inequality

\begin{equation*}
\lambda_i^{\mathcal P} (n) \ge N_{x_i} (n) + \{ L(n) \}
\end{equation*}
is proved as in the case of finite $\mathcal A$.

To prove the upper bound we will make a reduction of the general case
to the case of finite set of parameters by an operation of amalgamation.
Below we indicate sets of letters which we want to identify. We will
use a linear order on $\mathcal A$ such that these sets will be the intervals.
Our goal is to obtain a finite number of parameters in such a way that
the $K$ largest $\alpha$-parameters do not change and all parameters are distinct.

1)Assume that the number of $\alpha$-parameters in $\mathcal P$ is infinite.
Let us consider two cases

a) Suppose there exists $l \in \mathbb N$ such that:

\begin{equation}
\label{11}
\sum\nolimits_{i=l+1}^{\infty} \alpha_i < \alpha_K
\end{equation}
and for any $r \in \mathbb N$:

\begin{equation}
\label{12}
\sum\nolimits_{i=l+1}^{\infty} \alpha_i \ne \alpha_r.
\end{equation}
Then we identify letters $x_{l+1}, x_{l+2}, x_{l+3} \dots$.

b) If there is no $l$ satisfying \eqref{11},\eqref{12}
then it is easy to show that there exist
$l_1, m_1 \in \mathbb N$ such that

\begin{equation*}
\sum\nolimits_{i=l_1+1}^{\infty} \alpha_i = \alpha_r < \alpha_K \ \ \ \mbox{ for some } r \leq l_1
\end{equation*}
and the following is true

\begin{equation*}
\alpha_r > \sum_{i=l_1+1}^{l_1+m_1} \alpha_i > \alpha_{r+1},
\end{equation*}

\begin{equation*}
\sum\nolimits_{i=l_1+m_1+1}^{\infty} \alpha_i < \alpha_{l_1}.
\end{equation*}
In this case we identify $x_{l_1+1}, \dots, x_{l_1+m_1}$, and
(separately) $x_{l_1+m_1+1}$, $x_{l_1+m_1+2}$,$ \dots$.

There is a finite number of $\alpha$-parameters after these operations
of amalgamation. By $\alpha_R$ denote the minimal from them.

2) If there is an infinite number of $\beta$-parameters then we can choose
$l_2$ such that

\begin{equation}
\label{bet}
\sum\nolimits_{i=l_2+1}^{\infty} \beta_i < \alpha_R,
\end{equation}
and identify letters $y_{l_2+1}, y_{l_2+2}, y_{l_2+3} \dots$.
Note that after this operation of amalgamation a new $\alpha$-parameter arises.
By \eqref{bet} this parameter is less than the other $\alpha$-parameters.
Denote it by $\alpha_{R+1}$.

3) If $\gamma>0$ then we may choose
$m \in \mathbb N$ and $\delta_1>\delta_2> \dots > \delta_m \in \mathbb R$
such that
\begin{align*}
& \delta_1 + \delta_2 + \dots + \delta_m=0, \\
& \frac{\gamma}{m}+\delta_1 < \alpha_{R+1}, \\
& \frac{\gamma}{m} + \delta_m > 0.
\end{align*}
Divide $G$ into non-intersecting
intervals of lengths $\frac{\gamma}{m}+\delta_1$, $\frac{\gamma}{m}+\delta_2$, $\dots$,
$\frac{\gamma}{m}+\delta_m$ and identify points in these intervals.
We have $m$ new $\alpha$-parameters as a result of this operation.
It is clear that these parameters are pairwise distinct and are less than the previous ones.

Thus any set of parameters $\mathcal P$ can be reduced to a finite number
of parameters in such a way that the $K$ largest $\alpha$-parameters do not change.
Denote this finite set of parameters by $\mathcal P^{*}$. Recall that
$\lambda_i^{\mathcal P^{*}}(n)$ can be naturally defined on $(\mathcal A^n, \mu_n)$.

{\it Upper bound}.

1)First row.

By Lemma 2
\begin{equation*}
\lambda_1^{\mathcal P} (n) \le \lambda_1^{\mathcal P^*} (n).
\end{equation*}
There is only a finite number of parameters in
$\mathcal P^*$, therefore we can use the result of section 3.1

\begin{equation*}
\lambda_1^{\mathcal P^*} (n) \le N_{x_1} (n) + \{ L(n) \}.
\end{equation*}
Hence,
\begin{equation*}
\lambda_1^{\mathcal P}(n) \le N_{x_1} (n) + \{ L(n) \}.
\end{equation*}

2) Let us prove the theorem for the $l$-th row. By Lemma 2
\begin{equation*}
\lambda_1^{\mathcal P}(n) + \lambda_2^{\mathcal P}(n) + \dots + \lambda_l^{\mathcal P}(n) \le
\lambda_1^{\mathcal P^*}(n) + \lambda_2^{\mathcal P^*}(n) + \dots + \lambda_l^{\mathcal P^*}(n).
\end{equation*}
Using the results of section 3.1, we get
\begin{equation*}
\lambda_1^{\mathcal P^*}(n) + \lambda_2^{\mathcal P^*}(n) + \dots + \lambda_l^{\mathcal P^*}(n)
\le N_{x_1}(n) + \dots + N_{x_l} (n) + \{ L(n) \}.
\end{equation*}
From lower bound it follows
\begin{equation*}
\lambda_1^{\mathcal P}(n) + \lambda_2^{\mathcal P}(n) + \dots + \lambda_{l-1}^{\mathcal P}(n)
\ge N_{x_1} (n)+ \dots + N_{x_{l-1}}(n) + \{ L(n) \}.
\end{equation*}
Thus we have
\begin{equation*}
\lambda_l^{\mathcal P} (n) \le N_{x_l} (n) + \{L(n) \}.
\end{equation*}

We can derive the statement of the theorem for columns from the statement
of the theorem for rows by the same way as in the section 3.1. This completes
the proof of Theorem 2.

\subsection{The proof of Theorem 1}

By an easy computation of characteristic functions it can be shown that

\begin{multline*}
\eta_n := \Bigl(\frac{N_{x_1}(n) - \alpha_1 n}{\sqrt{n}}, \frac{N_{x_2}(n) - \alpha_2 n}{\sqrt{n}}, \dots,
\frac{N_{x_K}(n) - \alpha_K n}{\sqrt{n}}, \\ \frac{N_{y_1}(n) - \beta_1 n}{\sqrt{n}}, \dots, \frac{N_{y_L}(n) - \beta_L n}{\sqrt{n}} \Bigr)
\xrightarrow[Law]{} \eta  ,
\end{multline*}
where $\eta$ stands for a multidimensional random Gaussian variable with
zero mean and covariance matrix $C$ which is defined by the formula

\begin{equation*}
C=\begin{pmatrix}
\alpha_{1}-\alpha_{1}^2 & -\alpha_{1} \alpha_{2} & -\alpha_{1} \alpha_{3} & \ldots & -\alpha_{1} \alpha_{K} & -\alpha_{1} \beta_{1} & \ldots & -\alpha_{1} \beta_{L} \\
-\alpha_{2} \alpha_{1} & \alpha_{2}-\alpha_{2}^2 & -\alpha_{2} \alpha_{3} & \ldots & -\alpha_{2} \alpha_{K} & -\alpha_{2} \beta_{1} & \ldots & -\alpha_{2} \beta_{L} \\
\vdots & \vdots & \ddots & \vdots & \vdots  & \vdots & \vdots & \vdots  \\
-\alpha_{K} \alpha_{1} & -\alpha_{K} \alpha_{2} & \ldots & \ldots & \alpha_{K}-\alpha_{K}^2 & -\alpha_{K} \beta_{1} & \ldots & -\alpha_{K} \beta_{L} \\
-\beta_{1} \alpha_{1} & -\beta_{1} \alpha_{2} & \ldots & \ldots & -\beta_1 \alpha_{K} & \beta_{1} - \beta_1^2 & \ldots & -\beta_{1} \beta_{L} \\
\vdots & \vdots & \vdots & \vdots & \vdots  & \vdots & \ddots & \vdots  \\
-\beta_{L} \alpha_{1} & -\beta_{L} \alpha_{2} & \ldots & \ldots & -\beta_L \alpha_{K} & \ - \beta_1 \beta_L & \ldots & \beta_L-\beta_{L}^2 \\
\end{pmatrix}
\end{equation*}

It is easy to see that
\begin{equation*}
\psi_n := \frac{( \{ L(n) \}, \{ L(n) \}, \dots, \{ L(n) \} )}{\sqrt{n}} \xrightarrow[prob]{} 0
\end{equation*}
for any L-bounded sequences of random variables (denoted as $\{L(n) \}$).
It is well-known (see, e.g., \cite[Th 3.1]{bil}) that
$\eta_n^0 \xrightarrow[Law]{} \eta^0$ and
$\psi_n^0 \xrightarrow[prob]{} 0$ implies
\begin{equation*}
\eta_n^0 + \psi_n^0 \xrightarrow[Law]{} \eta^0.
\end{equation*}
Hence we obtain
\begin{multline*}
\Bigl(\frac{\lambda_1^{\mathcal P}(n) - \alpha_1 n}{\sqrt{n}}, \frac{\lambda_2^{\mathcal P}(n) - \alpha_2 n}{\sqrt{n}}, \dots,
\frac{\lambda_K^{\mathcal P}(n) - \alpha_K n}{\sqrt{n}}, \\ \frac{\lambda_1^{' \mathcal P}(n) - \beta_1 n}{\sqrt{n}}
\dots \frac{\lambda_L^{'\mathcal P}(n) - \beta_L n}{\sqrt{n}} \Bigr) = \eta_n + \psi_n \xrightarrow[Law]{} \eta.
\end{multline*}
This completes the proof of Theorem 1.

\end{document}